\documentclass[twoside,11pt]{article}

\usepackage{jmlr2e}

\usepackage{amssymb}
\usepackage{amsmath}
\usepackage{xcolor}
\usepackage{enumitem}
\usepackage{changepage}

\usepackage{natbib}

\numberwithin{theorem}{section} % important bit
\newtheorem{mydef}{Definition}

\ShortHeadings{Log-concavity of constrained multinomial likelihood functions}{Levin and Learned-Miller}
\firstpageno{1}

\begin{document}

\title{Log-Concavity of Multinomial Likelihood Functions\\ Under Interval Censoring Constraints\\ on Frequencies or Their Partial Sums}

\author{\name Bruce Levin \email bl6@columbia.edu \\
       \addr Department of Biostatistics\\
       Mailman School of Public Health \\
       Columbia University\\
       New York, NY 10032,  USA
       \AND
       \name Erik Learned-Miller \email elm@cs.umass.edu \\
       \addr Manning College of Information and Computer Sciences\\
       University of Massachusetts, Amherst\\
       Amherst, MA, 01002, USA}

\maketitle

\begin{abstract}
We show that the likelihood function for a multinomial vector observed under arbitrary interval censoring constraints on the frequencies or their partial sums is completely log-concave by proving that the constrained sample spaces comprise M-convex subsets of the discrete simplex.
\end{abstract}


\section{Introduction}
The theory of log-concavity has many applications in statistics and probability \cite[see, e.g.,][]{SaumardWellner2014}. The theory of Lorentzian polynomials as presented comprehensively by \citet{BrandanHuh2020} offers powerful tools for establishing log-concavity of homogeneous polynomials in several variables, such as the likelihood function from multinomial samples. In this note we establish log-concavity of the multinomial likelihood function when sampling is subject to one of two types of constraints: (a) arbitrary interval constraints on the components of the multinomial frequencies themselves; or (b) arbitrary interval constraints on the partial sums of those frequencies in a pre-specified ordering. The former arises naturally with interval-censoring arising due to quota sampling, for example, while the latter arises in a novel nonparametric confidence procedure for the mean of a non-negative distribution (see forthcoming work by second author). Our focus here will be on establishing the M-convexity of the constrained sample space under constraints of type (a) or (b). Theorems established by \citet{BrandanHuh2020} then imply the strong and complete log-concavity of the respective likelihood functions. We are pleased to bring their work to recognition in the statistics and probability community.

\section{Preliminaries and notation}
For given integers $m \geq 2$ and $n \geq 2$, let 
$$\Omega_{m, n}=\left\{x \in\{0, \ldots, n\}^{m}: x_{1}+\cdots+x_{m}=n\right\}$$
be the discrete simplex in $m$ variables of order $n$ and let 
$$\Delta_{m}=\left\{p \in \mathfrak{R}_{\geq 0}^{m}: p_{1}+\cdots+p_{m}=1\right\}$$
be the continuous simplex of dimension $m-1$. For $p \in \Delta_{m}$, let $X$ have a multinomial distribution with index $n$ and parameter $p$, $X \sim \operatorname{Mult}_{m}(n, p)$, with sample space $\Omega_{m, n}$. For arbitrary discrete intervals $R_{j}=\left\{l_{j}, \ldots, u_{j}\right\}$ with pre-specified integers $0 \leq l_{j} \leq u_{j} \leq n$, let $R=R(l, u)=\left(R_{1} \times \cdots \times R_{m}\right) \cap \Omega_{m, n} \quad$ where $\quad l=\left(l_{1}, \ldots, l_{m}\right) \quad$ and $u=\left(u_{1}, \ldots, u_{m}\right)$. We refer to $R$ as a rectangular subset of $\Omega_{m, n}$, even though $R$ does not have a rectangular geometric appearance when displayed in barycentric coordinates. When regarded as a function of $p$, the multinomial likelihood function for the rectangular event $[X \in R]$ is
\begin{eqnarray}
L(p \mid R)&=&P[X \in R \mid p]=\sum_{x \in R}{\binom{n}{x}} p^{x} \nonumber \\
\label{eq:likelihood}
&=&\sum_{x_{1}=l_{1}}^{u_{1}} \cdots \sum_{x_{m}=l_{m}}^{u_{m}} \frac{n !}{x_{1} ! \cdots x_{m} !} \prod_{j=1}^{m} p_{j}^{x_{j}}.
\end{eqnarray}
Inferences about $p$ may be based on (\ref{eq:likelihood}) when the multinomial frequencies are interval-censored by $R$. In (\ref{eq:likelihood}) we have used  ${\binom{n}{x}}=\frac{n !}{x_{1} ! \cdots x_{m} !}$ for the multinomial coefficient and $p^{x}=\prod_{j=1}^{m} p_{j}^{x_{j}}$ for $x \in \Omega_{m, n}$. The sums are equivalent since the multinomial coefficient is zero unless $x \in \Omega_{m, n}$.

Also for $x \in \Omega_{m, n}$, let $S_{k}=S_{k}(x)$ denote the partial sums $S_{k}(x)=x_{1}+\cdots+x_{k}$ for $k=1, \ldots, m-1$ in a given ordering of multinomial categories. Now for given integer $(m-1)$-vectors $l=\left(l_{1}, \ldots, l_{m-1}\right)$ and $u=\left(u_{1}, \ldots, u_{m-1}\right)$ with $0 \leq l_{1} \leq \cdots \leq l_{m-1} \leq n$, $0 \leq u_{1} \leq \cdots \leq u_{m-1} \leq n$, and $l_{k} \leq u_{k} \; 
 (k=1, \ldots, m-1)$, we define another subset of $\Omega_{m, n}$ which we call a partial sum rectangle, namely, $W=W_{m, n}(l, u)=\left\{x \in \Omega_{m, n}: l_{k} \leq S_{k} \leq u_{k}\right.$ for $\left.k=1, \ldots, m-1\right\}$. The likelihood function for $p$ given the partial sum rectangular event $\left[X \in W\right]$ is now
\begin{equation}
\label{eq:partialSum}
L(p \mid W)=P[X \in W \mid p]=\sum_{x \in W}{\binom{n}{x}} p^{x},
\end{equation}
though we cannot further expand the sum as on the right-hand side of (\ref{eq:likelihood}).

Rectangles and partial sum rectangles agree for $m=2$  but not for $m\geq 3$.  When $m=2$, the partial sum rectangle defined by $l_1\leq S_1 \leq u_1$ is just the rectangle $R=\{ l_1',...,u_1'\} \times \{ l_2',...,u_2' \}$  with  $l_1'=l_1$, $u_1'=u_1$, $l_2'=n-u_1$, and $u_2'=n-l_1$. 
For $m=3$, every partial sum rectangle is a rectangle but not conversely, i.e., there are rectangles $R$ such that any partial sum rectangle $W\supseteq R$  implies $W\supset R$.  It is not difficult to see that the partial sum rectangle defined by $l_j \leq S_j \leq u_j$ for $j=1,2$ agrees with the rectangle $R=\{l_1,...,u_1\}\times \{(l_2-u_1)^+,..., u_2-l_1\}\times\{n-u_2,...,n-l_2\}$.  However, for a partial sum rectangle $W=W_{m,n}(l,u)$ to contain a given rectangle $R=R_{m,n}(l',u')$, we must have $l_2\leq \min \{x_1+x_2: x\in R\}$  and $u_2\geq \max \{x_1+x_2: x\in  R\}$, but then $W$ may contain other $x\notin R$.  For example, with $m=3$, $n=8$, and $R$ given by $l'=(1,2,2)$ and $u'=(3,4,4)$, the minimum $S_2$ is $4$ and the maximum $S_2$  is $6$, so we must have $l_1=1$, $u_1=3$, $l_2\leq 4$, and $u_2\geq 6$, but any such $W$ admits the points $(3,1,4)$ and $(1,5,2) \notin R$.  

When $m>3$, there are also partial sum rectangles $W$ such that any rectangle $R\supseteq W$  implies $R\supset W$.  This is because for a rectangle  $R=R_{m,n}(l',u')$ to contain a given partial sum rectangle $W=W_{m,n}(l,u)$, we must have $l_j' \leq \min \{x_j: x\in W\}$ and $u_j'\geq \max\{x_j:x\in W\}$ for $j=1,...,m$, but then $R$ may contain other $x\notin W$.  
For example, with $m=4$, $n=5$, and $W$ given by $l=(2,3,4)$  and $u=(3,4,5)$, the minimum and maximum components are $(2,0,0,0)$  and $(3,2,2,1)$, respectively.  Taking these as $l'$  and $u'$, the rectangle $R_{m,n}(l',u')$  contains all points in $W$ but also the points $x=(2,0,2,1)$  and $x=(3,2,0,0)$  which violate the partial sum constraints.  Thus, rectangles and partial sum rectangles comprise different collections of subsets for $m>2$.

\citet{BrandanHuh2020} analyze a family of homogeneous polynomials of degree $n$ in $m$ variables $w_1,...,w_m$  which they call Lorentzian.  Below we briefly present their main results connecting Lorentzian polynomials with strongly log-concave and completely log-concave polynomials.  Note that other notions of log-concavity, such as log-concavity or ultra-log-concavity for discrete distributions or sequences are only tangentially related for our purposes, so will not be discussed here.  See \citet{BrandanHuh2020} or \citet{SaumardWellner2014} for those other notions.

The family of \textit{strictly} Lorentzian polynomials is given in Definition 2.1 of \citet{BrandanHuh2020} as homogeneous polynomials $f(w_1,...,w_m)$ of degree $n\geq 0$ with all \textit{positive} coefficients that satisfy the following recursive property: for $n=0$ or $n=1$, no further conditions; for $n=2$, the Hessian matrix $\{\partial f / \partial w_i \partial w_j\}$  for $i,j = 1,...,m$   is non-singular and has exactly one positive eigenvalue; and for $n>2$, the partial derivatives $\partial f/\partial w_i$  must be strictly Lorentzian of degree $n-1$  for each $i=1,...,m$.  [Our notation differs slightly from that of \citet{BrandanHuh2020}—for our number of variables $m$, they use $n$ and for our degree $n$, they use $d$.]  They then define Lorentzian polynomials as limits of strictly Lorentzian polynomials, which permits some monomial coefficients to be zero.  In a mathematical tour de force, they then prove the following equivalences (a)-(c), which characterize the Lorentzian polynomials quite nicely.  

(a) For an arbitrary $m$-vector with non-negative integer components, say $\gamma=(\gamma_1,...,\gamma_m)$, let   $\partial^\gamma = \partial^{\gamma_1+...+\gamma_m}/\partial w_1^{\gamma_1}...\partial w_m^{\gamma_m}$ be the mixed-derivative operator.  A polynomial $f$ in $m$ variables with non-negative coefficients is said to be \textit{strongly log-concave} if $\partial^\gamma f$  is identically zero or log-concave on the positive orthant $\{w_1>0,...,w_m>0\}$  for all  $\gamma$ \citep{Gurvits2009}.  \textit{Then a degree $n$ homogeneous polynomial is Lorentzian if and only if it is strongly log-concave} \citep[Theorem 2.30]{BrandanHuh2020}.

(b) For a set of $m$-vectors $a_i=(a_{i1},...,a_{im})$  with non-negative components, let $D_i$  be the differential operator $D_i=\sum_{j=1}^m a_{ij}\; \partial / \partial w_j$   for $i=1,...,k$.  A polynomial $f$ in $m$ variables is said to be \textit{completely log-concave} if $f$ is log-concave and $D_1...D_k f$  is non-negative and log-concave on the positive orthant for any $k\geq 1$  and any $a_1,...,a_k$  (\citet{anari2018logconcave}).  \textit{Then a homogeneous polynomial of degree $n$ is Lorentzian if and only if it is completely log-concave} \citep[Theorem 2.30]{BrandanHuh2020}.

(c) Let $f(w)=\sum_\gamma c_\gamma w^\gamma$   be a polynomial in $m$ variables with non-negative coefficients  $c_\gamma$, where $w^\gamma = w_1^{\gamma_1}... w_m^{\gamma_m}$.  The \textit{support} of $f$ is defined as the subset $\{\gamma: c_\gamma > 0\}$ of monomials with positive coefficients.  Also, a subset $C\subseteq\Omega_{m,n}$  is said to be \textit{M-convex} if it satisfies the following “exchangeability” condition: for any $\alpha,\beta\in C$  and any index $i$ satisfying $\alpha_i>\beta_i$, there is an index $j$ satisfying  $\alpha_j<\beta_j$ and $\alpha-e_i+e_j\in C$, where $e_i$  and $e_j$   are the standard unit vectors in $\mathfrak{R}^m$.  There are other equivalent conditions; see, e.g., \citet{Murota2003} or \citet{BrandanHuh2020}, p.9.  The stated condition is most convenient for our purposes.  \textit{Then a degree $n$ polynomial $f$ with non-negative coefficients is Lorentzian if and only if the support of $f$  is M-convex and the Hessian of $\partial^\gamma f$  has at most one positive eigenvalue for every $\gamma\in \Omega_{m,n-2}$} \citep[p. 22]{BrandanHuh2020}.

Our interest in log-concavity for likelihood functions (1) and (2) arises for two reasons.  First, it ensures that (1) is unimodal in $p$, which greatly simplifies maximum likelihood estimation of $p$ under interval censoring constraints.  Second, log-concavity of (2) greatly simplifies the task of locating worst-case error sets which are used to guarantee coverage probabilities in the above-mentioned non-parametric confidence procedure.  The statistically interesting and insightful Theorem 3.10 of \citet{BrandanHuh2020} shows that \textit{(1) and (2) are Lorentzian polynomials of order $n$ for $p\in \Delta_m$  if and only if $R$ or $W$, respectively, are M-convex subsets}.  As indicated above, this implies both the strong and complete log-concavity of (1) and (2).  In our Theorem 3.1 below we show that any rectangle $R_{m,n}(l,u)$  is M-convex and in Theorem 3.2 we show that any partial sum rectangle $W_{m,n}(l,u)$  is M-convex.

\section{Proofs of M-convexity}

\begin{theorem}
\label{thm:rectangle}
Any rectangle $R=R(l, u) \subseteq \Omega_{m, n}$ is $M$-convex.
\end{theorem}
\begin{proof}
Suppose $\alpha, \beta \in R$ and index $i$ satisfies $\alpha_{i}>\beta_{i}$. Then we must have $l_{i}<\alpha_{i}$, for if not, then $\alpha \in R$ implies $\alpha_{i}=l_{i}$, but then $\beta \in R$ implies $\beta_{i} \geq l_{i}=\alpha_{i}$, a contradiction. Then we already have $l_{i} \leq \alpha_{i}-1 \leq u_{i}$. Now we claim that there must be another index $j \neq i$ such that $\alpha_{j}<\beta_{j}$. For if not, then $\alpha,\beta \in R \subset \Omega_{m,n}$ implies $n=\sum_{k\neq i}\beta_k + \beta_i\leq \sum_{k\neq i} \alpha_k + \beta_i < \sum_{k\neq i}\alpha_k + \alpha_i=n,$ a contradiction. So there exists $j\neq i$ such that $\alpha_j < \beta_j \leq u_j$
and for any such $j$ we have $l_{j} \leq \alpha_{j}+1 \leq u_{i}$, whence $\alpha-e_{i}+e_{j} \in R$.
\end{proof}
The proof of M-convexity for partial sum rectangles is rather more subtle. We shall find the following definition useful.

\begin{mydef}
\label{def:feasible}
    Let $W=W_{m, n}(l, u)$ be a given partial sum rectangle, let $\alpha, \beta \in W$, and let $i$ be such that $\alpha_{i}>\beta_{i}$. An index $j \in\{1, \ldots, m\}$ is feasible if
\begin{enumerate}[label=(\roman*)]
\item $j<i$ and $\alpha_{j}<\beta_{j}$ with $S_{k}(\alpha)<u_{k}$ for each $j \leq k<i$, or
\item $j>i$ and $\alpha_{j}<\beta_{j}$ with $S_{k}(\alpha)>l_{k}$ for each $i \leq k<j$.
\end{enumerate}
\end{mydef}
Feasible indices are those for which it is possible for $\alpha^{\prime}=\alpha-e_{i}+e_{j}$ to satisfy the requirements for $M$ convexity of $W$.

\begin{theorem}
\label{thm:partialSUm}
Any partial sum rectangle $W=W_{m, n}(l, u) \subseteq \Omega_{m, n}$ is $M$-convex.
\end{theorem}
\begin{proof} 
To establish the theorem, we demonstrate two lemmas. The first shows that at least one feasible index exists and the second shows that for an appropriately selected feasible index $j$, $\alpha^{\prime}=\alpha-e_{i}+e_{j} \in W$. 

\begin{lemma}
\label{lem:feasibleIndex}    For any partial sum rectangle $W=W(l, u) \subseteq \Omega_{m, n}$ with $\alpha, \beta \in W$ and $i \in\{1, \ldots, m\}$ such that $\alpha_{i}>\beta_{i}$, there exists a feasible index $j$.
\end{lemma}
\begin{proof}
We first consider the simpler boundary cases, (a) $i=1$ and (b) $i=m$, and conclude with (c) $1<i<m$.

\textbf{(a)} Suppose $i=1$. There must exist an index $j$ with $1<j \leq m$ such that $\alpha_{j}<\beta_{j}$, for if $\alpha_{k} \geq \beta_{k}$ for each $1<k \leq m$, then $\alpha_{i}>\beta_{i}$ implies $n=\alpha_{1}+\cdots+\alpha_{m}>\beta_{1}+\cdots+\beta_{m}=n$, contradiction. If $j=2$ is among those indices with $\alpha_{j}<\beta_{j}$, we are done as $j=2$ already satisfies feasibility Definition~(\ref{def:feasible})(ii), because we must have $S_{1}(\alpha)>l_{1}$, else $\alpha_{1}=S_{1}(\alpha)=l_{1} \leq S_{1}(\beta)=\beta_{1}$, contradiction. So suppose $j>2$ and consider the least such index, $j=\min \left\{k>1: \alpha_{k}<\beta_{k}\right\}$. Since $j>1$ is least, we have $\alpha_{k} \geq \beta_{k}$ for $1<k<j$. Then $\alpha_{i}>\beta_{i}$ implies $S_{k}(\alpha)>S_{k}(\beta) \geq l_{k}$ since $\beta \in W$. Thus, $\alpha_{j}<\beta_{j}$ with $S_{k}(\alpha)>l_{k}$ for $i \leq k<j$, so $j$ is feasible under Definition~(\ref{def:feasible})(ii).

\textbf{(b) }Next, suppose $i=m$. There must exist an index $j$ with $1 \leq j<m$ such that $\alpha_{j}<\beta_{j}$, for the same reason as in (a). If $j=m-1$ is among those indices with $\alpha_{j}<\beta_{j}$, we are done as $j=m-1$ already satisfies feasibility Definition~(\ref{def:feasible})(i), because we must have $S_{m-1}(\alpha)<u_{1}$, else we have the contradiction $\alpha_{m}=n-S_{m-1}(\alpha)=n-u_{m-1} \leq n-S_{m-1}(\beta)=\beta_{m}$. So suppose $j<m-1$ and consider the greatest such index, $j=\max \left\{k<m: \alpha_{k}<\beta_{k}\right\}$. Since $j<m$ is greatest, we have $\alpha_{k} \geq \beta_{k}$ for $j<k<m$. Then $\alpha_{m}>\beta_{m}$ implies $S_{k}(\alpha)=n-\left(\alpha_{k+1}+\cdots+\alpha_{m}\right)<n-\left(\beta_{k+1}+\cdots+\beta_{m}\right)=S_{k}(\beta) \leq u_{k}$ since $\beta \in W$. Thus, $\alpha_{j}<\beta_{j}$ with $S_{k}(\alpha)<u_{k}$ for $j \leq k<i$, so $j$ is feasible under Definition~(\ref{def:feasible})(i).

\textbf{(c)} Now suppose $1<i<m$. We show that if there is no feasible $j$ under Definition~(\ref{def:feasible})(i) then there exists a feasible $j$ under Definition~(\ref{def:feasible})(ii), and conversely. Suppose there is no feasible $j<i$. Then for each $1 \leq k<i$, either ($\dagger) S_{k}(\alpha)=u_{k}$ or $(\ddagger) S_{k}(\alpha)<u_{k}$ but $\alpha_{k} \geq \beta_{k}$. We have the following six consequences (C1)-(C6).

\vspace{.1in}

 \noindent \textbf{(C1)} $S_{k}(\alpha) \geq S_{k}(\beta)$ for each $1 \leq k<i$.
 
\begin{adjustwidth}{\parindent}{}
By induction on $k$. If ($\dagger)$ holds, then $S_{1}(\alpha)=u_{1} \geq S_{1}(\beta)$ since $\beta \in W$. If ($\ddagger$) holds, then already $\alpha_{1} \geq \beta_{1}$. So assume that $S_{k-1}(\alpha) \geq S_{k-1}(\beta)$. If ($\dagger$) holds, then $S_{k}(\alpha)=u_{k} \geq S_{k}(\beta)$ as before, while if $(\ddagger)$ holds, then $S_{k}(\alpha)=S_{k-1}(\alpha)+\alpha_{k} \geq S_{k-1}(\beta)+\alpha_{k} \geq S_{k-1}(\beta)+\beta_{k}=S_{k}(\beta)$. The first inequality is by the inductive hypothesis and the second is by $(\ddagger)$.
\end{adjustwidth}

 \noindent \textbf{(C2)} $S_{i}(\alpha)>S_{i}(\beta).$

\begin{adjustwidth}{\parindent}{}
This is because $S_{i}(\alpha)=S_{i-1}(\alpha)+\alpha_{i} \geq S_{i-1}(\beta)+\alpha_{i}>S_{i-1}(\beta)+\beta_{i}=S_{i}(\beta)$ by (C1) and the assumption $\alpha_{i}>\beta_{i}$. 
\end{adjustwidth}

 \noindent \textbf{(C3)} $S_{i}(\alpha)>l_{i}$ or equivalently, $S_{i}(\alpha)-1 \geq l_{i}$.

\begin{adjustwidth}{\parindent}{}
By $(\mathrm{C} 2)$, since $S_{i}(\beta) \geq l_{i}$.
\end{adjustwidth}

 \noindent \textbf{(C4)} It follows that there must exist an index $j>i$ such that $\alpha_{j}<\beta_{j}$.

\begin{adjustwidth}{\parindent}{}
For if $\alpha_{j} \geq \beta_{j}$ for $k>i$, then $\alpha_{i+1}+\cdots+\alpha_{m} \geq \beta_{i+1}+\cdots+\beta_{m}$ which implies the contradiction $S_{i}(\alpha)=n-\left(\alpha_{i+1}+\cdots+\alpha_{m}\right) \leq n-\left(\beta_{i+1}+\cdots+\beta_{m}\right)=S_{i}(\beta)$ by $(\mathrm{C} 2)$. Thus we can and do take $j$ to be the least such index, in which case $\alpha_{k} \geq \beta_{k}$ for $i<k<j$.
\end{adjustwidth}

\noindent  \textbf{(C5)} $S_{k}(\alpha)>S_{k}(\beta)$ for each $i \leq k<j=\min \left\{k>i: \alpha_{k}<\beta_{k}\right\}$.

\begin{adjustwidth}{\parindent}{}
By induction on $k$. The initial case $k=i$ is (C2). Assuming $S_{k-1}(\alpha)>S_{k-1}(\beta)$, $S_{k}(\alpha)=S_{k-1}(\alpha)+\alpha_{k}>S_{k-1}(\beta)+\beta_{k}=S_{k}(\beta)$ by the inductive hypothesis and the last assertion of (C4).
\end{adjustwidth}

 \begin{adjustwidth}{\parindent}{}
\noindent \hspace{-\parindent} \textbf{(C6) }Since $S_{k}(\beta) \geq  l_{k}$, it follows that $S_{k}(\alpha)>l_{k}$ or equivalently, $S_{k}(\alpha)-1 \geq l_{k}$ for each $i \leq k<j$.
\end{adjustwidth}

 We have thus established $j>i$ is feasible under Definition~(\ref{def:feasible})(ii). The proof that there exists a feasible $j<i$ under Definition~(\ref{def:feasible})(i) when there is no feasible $j>i$ under Definition~(\ref{def:feasible})(ii) is  entirely analogous by symmetry and will be omitted for brevity.
\end{proof}
Next we prove that the feasible indices identified in Lemma~\ref{lem:feasibleIndex} satisfy $\alpha^{\prime}=\alpha-e_{i}+e_{j} \in W$, as required for showing $W$ is $M$-convex.

\begin{lemma}
\label{lem:f2}
For feasible $j>i$ with $j=\min \left\{k>i: \alpha_{k}<\beta_{k}\right\}, \alpha^{\prime}=\alpha-e_{i}+e_{j} \in W$. The same holds for feasible $j<i$ with $j=\max \left\{k<i: \alpha_{k}<\beta_{k}\right\}$.
\end{lemma}

\noindent {\bf Proof}\hspace{.1in}
We verify the requisite inequalities $l_{k} \leq S_{k}\left(\alpha^{\prime}\right) \leq u_{k}$ for $\alpha^{\prime}$ when $j>i$. The verification when $j<i$ is entirely analogous and will be omitted. We have six further consequences.

\noindent \textbf{(C7)} $l_{k} \leq S_{k}\left(\alpha^{\prime}\right) \leq u_{k}$ for $1 \leq k<i$.

\begin{adjustwidth}{\parindent}{}
By definition of $\alpha^{\prime}, \alpha_{k}^{\prime}=\alpha_{k}$ for $1 \leq k<i$, so $S_{k}\left(\alpha^{\prime}\right)=S_{k}(\alpha)$ for such $k$.
\end{adjustwidth}

\noindent \textbf{(C8)} $l_{i} \leq S_{i}\left(\alpha^{\prime}\right) \leq u_{i}$.

\begin{adjustwidth}{\parindent}{}
For $S_{i}\left(\alpha^{\prime}\right)=S_{i-1}\left(\alpha^{\prime}\right)+\alpha_{i}^{\prime}=S_{i-1}(\alpha)+\alpha_{i}-1=S_{i}(\alpha)-1 \geq l_{i} \quad$ by $\quad$ (C7) and (C3). Obviously, $S_{i}\left(\alpha^{\prime}\right) \leq u_{i}$ since $S_{i}(\alpha) \leq u_{i}$.
\end{adjustwidth}

 \noindent\textbf{(C9)} $S_{k}\left(\alpha^{\prime}\right)=S_{k}(\alpha)-1$ for $i<k<j$. 
 
 \begin{adjustwidth}{\parindent}{}
By induction on $k$. The case $k=i$ follows from the first line of the proof of (C8). Assuming $S_{k-1}\left(\alpha^{\prime}\right)=S_{k-1}(\alpha)-1, S_{k}\left(\alpha^{\prime}\right)=S_{k-1}\left(\alpha^{\prime}\right)+\alpha_{k}^{\prime}=S_{k-1}(\alpha)-1+\alpha_{k}=S_{k}(\alpha)-1$ by the inductive hypothesis and the definition of $\alpha^{\prime}$ wherein only elements $i$ and $j$ differ from those of $\alpha$.
\end{adjustwidth}

\noindent \textbf{(C10)} $l_{k} \leq S_{k}\left(\alpha^{\prime}\right) \leq u_{k}$ for $i<k<j$.

\begin{adjustwidth}{\parindent}{}
For from (C9) and (C6), $S_{k}\left(\alpha^{\prime}\right)=S_{k}(\alpha)-1>l_{k}-1$, whence $S_{k}\left(\alpha^{\prime}\right) \geq l_{k}$.
Obviously, $S_{k}\left(\alpha^{\prime}\right) \leq u_{k}$ by (C9) since $S_{k}(\alpha) \leq u_{k}$.
\end{adjustwidth}

\noindent \textbf{(C11)} $S_{j}\left(\alpha^{\prime}\right)=S_{j}(\alpha)$.

\begin{adjustwidth}{\parindent}{}
For $S_{j}\left(\alpha^{\prime}\right)=S_{j-1}\left(\alpha^{\prime}\right)+\alpha_{j}^{\prime}=\left\{S_{j-1}(\alpha)-1\right\}+\left(\alpha_{j}+1\right)=S_{j}(\alpha)$.
\end{adjustwidth}

\noindent \textbf{(C12)} If $j<m-1$, then we also have $S_{k}\left(\alpha^{\prime}\right)=S_{k}(\alpha)$ for each $j \leq k<m$.

\begin{adjustwidth}{\parindent}{}
By definition of $\alpha^{\prime}, \quad \alpha_{k}^{\prime}=\alpha_{k}$ for $j<k \leq m$, and since $S_{j}\left(\alpha^{\prime}\right)=S_{j}(\alpha)$ by (C11), we have $S_{k}\left(\alpha^{\prime}\right)=S_{k}(\alpha)$ for each $j<k<m$. Thus with (C11), $l_{k} \leq S_{k}\left(\alpha^{\prime}\right) \leq u_{k}$ for $j \leq k<m$.
\end{adjustwidth}

\vspace{.1in}

Therefore, $\alpha^{\prime} \in W$ with $\alpha_{j}<\beta_{j}$, so $W$ is $M$-convex.  This concludes the proofs of Lemma~\ref{lem:f2} and Theorem~\ref{thm:partialSUm}.
\end{proof}
\vspace{-.1in}
\begin{remark}\normalfont
    Note that by $(\mathrm{C} 3), S_{i}(\alpha)>l_{i}$ must hold under the assumption that no feasible $j<i$ exists, but $S_{i}(\alpha)=l_{i}$ is allowed when there are feasible $j<i$ [as there must be by Lemma~\ref{lem:f2}, since no $j>i$ is feasible if $S_{i}(\alpha)=l_{i}$ under Definition~(\ref{def:feasible})(ii)]. Similarly, $S_{j}(\alpha)<u_{j}$ must hold if there are no feasible $j>i$, but $S_{j}(\alpha)=u_{j}$ is allowed when there are feasible $j>i$ [as there must be by Lemma~\ref{lem:f2}, since $j<i$ is not feasible if $S_{j}(\alpha)=u_{j}$ under Definition~(\ref{def:feasible})(i)].
\end{remark}

\begin{remark} \normalfont
In the proof of case (c) of Lemma~\ref{lem:feasibleIndex}, under the assumption of no feasible $j<i$, we did not need the inequality $S_{j}(\alpha)>l_{j}$, which isn't even necessarily true. Even if $S_{j}(\alpha)=l_{j}$, as $k$ increases from $i$, the partial sums $S_{k}\left(\alpha^{\prime}\right)$ always ``adjust up'' to the constraint $l_{j} \leq S_{j}\left(\alpha^{\prime}\right) \leq u_{j}$ with $\alpha_{j}^{\prime}=\alpha_{j}+1$. Similarly, under the assumption of no feasible $j>i$, we do not need the inequality $S_{i}(\alpha)>u_{i}$, which isn't necessarily true. Even if $S_{i}(\alpha)=u_{i}$, as $k$ increases from $j$, the partial sums $S_{k}\left(\alpha^{\prime}\right)$ always ``adjust down'' to the constraint $l_{i} \leq S_{i}\left(\alpha^{\prime}\right) \leq u_{i}$ with $\alpha_{i}^{\prime}=\alpha_{i}-1$.
\end{remark}
\section{Acknowledgment}
The authors wish to thank Cynthia Vinzant for very helpful conversations on complete log-concavity and Lorentzian polynomials.

\bibliography{logcon}
\end{document}